\documentclass[12pt,usenames,dvipsnames]{article}
\usepackage{hyperref}
\usepackage{cmap}
\usepackage[T1]{fontenc}\usepackage[cp1251]{inputenc}            
\usepackage[english]{babel}
\usepackage{multicol, wrapfig}
\usepackage[left=1in,right=1in,top=1in,bottom=1.5in]{geometry} 
\usepackage{amssymb,amsmath, amsthm, amscd,ifthen}
\usepackage{graphicx}
\usepackage{enumitem}
\usepackage[svgnames]{xcolor}
\usepackage{vwcol}
\usepackage{tikz}
\usetikzlibrary{patterns,decorations.text,positioning,arrows,shapes,decorations.markings}
\tikzstyle{vertex}=[circle,draw=black,fill=black,inner sep=0,minimum size=0.2cm,text=white,font=\footnotesize]

\newcommand{\R}{{\mathbb R}}

\newtheorem*{thm*}{Theorem}
\newcommand{\ff}{{\mathcal F}}
\newcommand{\aaa}{{\mathcal A}}

\newtheorem*{cla*}{Claim}
\newcommand{\bb}{{\mathcal B}}

\newcommand{\hh}{\mathcal H}
\newtheorem{thm}{Theorem}

\newtheorem{opr}{Definition}
\newtheorem{lem}[thm]{Lemma}

\date{}

\newtheorem{prop}[thm]{Proposition}

\title{Families of vectors without antipodal pairs}
\author{Peter Frankl, Andrey Kupavskii
}
\date{}

\begin{document}
\maketitle
\begin{abstract} Some Erd\H os-Ko-Rado type extremal properties of families of vectors from $\{-1,0,1\}^n$ are considered.
\end{abstract}
\vskip+0.1cm\noindent

\section{Introduction}
The standard $n$-cube is formed by all vectors $v=(v_1,\ldots, v_n)$ with $v_i\in\{0,1\}$. Setting $F(v):=\{i:v_i=1\}$ is a natural way to associate a subset of $[n]:=\{1,\ldots, n\}$ with a vertex of the $n$-cube. This association has proved very useful in tackling various problems in discrete geometry. In particular, intersection theorems concerning finite sets were the main tool in proving exponential lower bounds for the chromatic number of $\R^n$ and disproving Borsuk's  conjecture in high dimensions (cf. \cite{FW}, \cite{KK}).

In this short note we consider $(0,\pm 1)$-vectors, that is, vectors $v = (v_1,\ldots, v_n)$, where each $v_i$ is $0,1,$ or $-1$. Probably the first non-trivial extremal result concerning these objects was a result of Deza and the first author \cite{DF} showing that in a certain situation one can prove the same best possible upper bound for $(0,\pm 1)$-vectors as for the restricted case of $(0,1)$-vectors.

Raigorodskii \cite{Rai1} and others (cf, e.g., \cite{PR}, \cite{K}) have used a similar approach to improve the bounds for the above-mentioned and related discrete geometry problems, obtained via $(0,1)$-vectors, by considering $(0,\pm 1)$-vectors.

Motivated by such results we propose to investigate the following problem. Let $k\ge l\ge 1$ be integers and let $V(n,k,l)$ denote the set of all $(0,\pm 1)$-vectors of length $n$ and having exactly $k$ coordinates equal to $+1$ and $l$ coordinates equal to $-1$. Note that $$|V(n,k,l)|= {n\choose k}{n-k\choose l} = {n\choose k+l}{k+l\choose l}.$$

For two vectors let $\langle v, w \rangle$ denote their scalar product: $\langle v,w\rangle = \sum_{i=1}^n v_iw_i$.

If $v,w\in V(n,k,l)$, then they possess altogether $2l$ coordinates equal to $-1$. Thus $$\langle v,w\rangle \ge -2l.$$
If $\langle v,w\rangle =-2l$, then we call these two vectors \underline{antipodal}. Note that for a fixed $v\in V$ there are ${k\choose l}{n-k-l\choose k-2l}$ antipodal vectors $w\in V$. To avoid trivialities, we assume in what follows that $n\ge 2k$.

\textbf{Example 1 } Let $\mathcal G\subset V(n,k,l)$ consist of those vectors whose \underline{last} non-zero coordinate is a $-1$. Then $|\mathcal G| = {n\choose k+l}{k+l-1\choose l-1}$, and it is easy to see that $\mathcal G$ contains no two antipodal vectors.

The purpose of this note is to prove the following two theorems.

\begin{thm}\label{thm1} Suppose that $\ff\subset V(n,k,l)$ does not contain two antipodal vectors. Then
\begin{equation}\label{eq0}|\ff|\le {n\choose k+l}{k+l-1\choose l-1}+{n\choose 2l}{2l\choose l}{n-2l-1\choose k-l-1}.\end{equation}
\end{thm}
Note that the last term of \eqref{eq0} is $O(n^{k+l-1})$. We also put ${n-2l-1\choose k-l-1}:=0$ for $k=l$. Thus \eqref{eq0} shows that Example 1 is asymptotically best possible. \\

\textbf{Example 2 } Let $\mathcal E\subset V(n,k,l)$ consist of those vectors whose \underline{first}  coordinate is a $1$. Then $|\mathcal E| = {n-1\choose k+l-1}{k+l-1\choose k-1} = \frac k n|V(n,k,l)|$, and $\mathcal E$ contains no two antipodal vectors.

\begin{thm}\label{thm2} Suppose that $\ff\subset V(n,k,l)$ does not contain two antipodal vectors. If $2k\le n\le 3k-l$, then \begin{equation}\label{eq1} |\ff|\le \frac k n|V(n,k,l)|.\end{equation}\end{thm}

Theorem~\ref{thm2} shows that Example 2 is best possible for $2k\le n\le 3k-l$. We note that the case $n\le 2k$ can be easily reduced to sets setting and the Erd\H os-Ko-Rado theorem (see below).
Let us also mention that in \cite{FK1} we gave the complete solution for the case $l=1$:

\begin{thm*}[Frankl, Kupavskii \cite{FK1}] Suppose that $\ff\subset V(n,k,1)$ does not contain two antipodal vectors. Then  one has
\begin{align*}|\ff|&\le  k{n-1\choose k} \text{\ \ \qquad \qquad \quad \ \ \ \ \ for\ }2k\le n\le k^2,\\
|\ff| \le k{k^2-1\choose k}+{k^2\choose k}&+{k^2+1\choose k}+\ldots +{n-1\choose k}  \text{\ \ \ \ \ \ \ \ for\ } n> k^2.
\end{align*}
Both inequalities are best possible.
\end{thm*}

The proof of Theorem~\ref{thm1} is rather short, but it relies on some classical results in extremal set theory.

\begin{opr} Two families $\aaa, \bb$ of finite sets are called \underline{cross-intersecting}, if for all $A\in\aaa$, $B\in \bb$ one has $A\cap B\ne \emptyset$. For the case $\aaa=\bb$ we use the term \underline{intersecting}.\end{opr}

 For $0\le k\le n$ let ${[n]\choose k}$ denote the collection of all $k$-subsets of $\{1,\ldots, n\}$.

\begin{thm*}[Erd\H os-Ko-Rado \cite{EKR}] suppose that $n\ge 2k>0$, and the family $\aaa\subset {[n]\choose k}$ is intersecting. Then
\begin{equation}\label{eq2} |\mathcal A|\le {n-1\choose k-1}.\end{equation}
\end{thm*}
As Daykin \cite{D} observed, \eqref{eq2} can be deduced from the Kruskal-Katona Theorem (\cite{Kr}, \cite{Ka}). the same approach yields the following version of \eqref{eq2} for cross-intersecting families.

\begin{prop}\label{prop1} Let $a,b,m$ be integers, $m\ge a+b$. Suppose that $\mathcal A\subset {[m]\choose a}$ and $\bb\subset {[m]\choose b}$ are cross-intersecting. Then either $|\aaa|\le {m-1\choose a-1}$ or $|\bb|\le {m-1\choose b-1}$ hold.
\end{prop}

Note that stronger versions of this proposition were proved by Pyber \cite{P}, Matsumoto and Tokushige \cite{MT}, and the authors of this note \cite{FK3}.

\section{The proof of Theorem~\ref{thm1}}

Let $\ff$ be our family of vectors. For a $(0,\pm 1)$-vector $v=(v_1,\ldots, v_n)$ let $S(v)$ denote its \underline{support}, i.e., \begin{equation*}S(v):=\{i\in[n]: v_i\ne 0\}.\end{equation*} Define also \begin{align*}S_+(v):=\{i:v_i= +1\},\\ S_-(v):=\{i:v_i=-1\}.\end{align*} Obviously, $S(v) = S_+(v)\sqcup S_-(v)$.

Also, for $k\ge l$ two vectors $v,w$ satisfy $(v,w)=-2l$ iff $S_-(v)\subset S_+(w), S_-(w)\subset S_+(v)$ and $S_+(v)\cap S_+(w)=\emptyset$ hold simultaneously.

Our assumption is that no such pair $v,w$ exist in $\ff$. For a pair $A,B$ of disjoint $l$-element sets we define $\ff(A,B)$ to be the family of those $(k-l)$-element sets $C$ that the vector $u$ defined by $S_+(u) = A\sqcup C$, $S_-(u) = B$ is in $\ff$.

\begin{lem}\label{lem1}
For disjoint $l$-subsets $A,B\subset [n]$ the two families $\ff(A,B)$ and $\ff(B,A)$ are cross-intersecting.
\end{lem}
\begin{proof}
Suppose the contrary and let $C\in \ff(A,B), D\in \ff(B,A)$ be disjoint $(k-l)$-sets. Then the vectors $v,w$ determined by $S_+(v) = A\sqcup C$, $S_-(v) = B$, $S_+(w) = B\sqcup D$, $S_-(w) = A$ are both in $\ff$. However, $\langle v, w\rangle = -2l$, a contradiction.
\end{proof}

This lemma and Proposition \ref{prop1} motivate the following procedure. For all ${n\choose 2l}{2l\choose l}$ choices of a pair of disjoint $l$-sets $A$ and $B$, if $|\ff(A,B)|\le {n-2l-1\choose k-l-1}$, then delete from $\ff$ {\it all} vectors $v$ with $S_-(v) = B$, $A\subset S_+(v)$.

Let $\ff'$ be the collection of remaining vectors and note:
\begin{equation}\label{eq3} |\ff'|\ge |\ff|-{n\choose 2l}{2l\choose l}{n-2l-1\choose k-l-1}.\end{equation}
Let us fix now a $(k+l)$-element set $T\subset [n]$ and consider the family $\bb\subset {T\choose l}$ defined as follows:
$$\bb:= \Big\{B\in {T\choose l}: \exists v\in\ff', S(v) = T, S_-(v) = B\Big\}.$$

\begin{lem}\label{lem2} The family $\mathcal B$ is intersecting. \end{lem}
\begin{proof} Suppose for contradiction that $A,B\in \bb$ are disjoint. By the definition of $\bb$  there are $u,v\in \ff'$ satisfying $S_-(u) = A$, $S_-(v) = B$, $S(u)=S(v) = T$. This implies $A\subset S_+(v), B\subset S_+(u)$.

Since both $u$ and $v$ survived the deletion process, we have
\begin{align*}|\ff(A,B)|>& {n-2l-1\choose k-l-1},\\ |\ff(B,A)|>& {n-2l-1\choose k-l-1}.\end{align*}

However, Proposition~\ref{prop1} shows that $\ff(A,B)$ and $\ff(B,A)$ are {\it not} cross-intersecting. This contradicts Lemma~\ref{lem1}.
\end{proof}

Since $k\ge l$, the Erd\H os-Ko-Rado Theorem implies \begin{equation*}\label{eq4}|\bb|\le {k+l-1\choose l-1}.\end{equation*}
Consequently,
\begin{equation*} |\ff'|\le {n\choose k+l}{k+l-1\choose l-1}.\end{equation*}
Combining with \eqref{eq3}, the inequality \eqref{eq0} follows.

\section{The proof of Theorem~\ref{thm2}}
The proof is based on the application of the general Katona's circle method \cite{Ka2} to $\mathcal V(n,k,l)$. Consider the following subfamily $\mathcal H$ of $\mathcal V(n,k,l)$:
$$\mathcal H := \{\mathbf v = (v_1,\ldots, v_n): \exists i\in[n]: \ v_i = \ldots = v_{i+k-1} = 1,\ v_{i-k} = \ldots = v_{i-k+l-1} = -1\}.$$
We remark that all indices are written modulo $n$. Note that $|\hh| = n$. For any permutation $\sigma$ of $[n]$ we  define $\hh(\sigma):=\{\sigma(H): H\in \hh\}$.

Take a family $\ff$ with no two antipodal vectors.

\begin{lem}\label{lem3} For any permutation $\sigma$ we have $|\hh(\sigma)\cap \ff|\le k.$
\end{lem}
\begin{proof}
Denote by $\mathcal F'\subset{[n]\choose k}$ the family $\{S_+(F): F\in \ff\}$, and, similarly,  $\mathcal H':=\{S_+(H): H\in \hh(\sigma)\}$.

We claim that $\mathcal H'\cap \mathcal F'$ is an intersecting family. Assume that there are two sets $F'_1,F'_2\in \mathcal H'\cap \mathcal F'$, that are disjoint. W.l.o.g., $F'_2 = [k+1,2k]$. Then $F'_1$ is obliged to contain $[1,l]$, since any cyclic interval of length $k$ in $[n]\setminus [k+1,2k]$ contains $[1,l]$, provided that $n\le 3k-l$.

We conclude that the corresponding vector $v_1\in \mathcal F\cap \mathcal H$ satisfies $S_+(v)\supset [1,l]$. At the same time, by the definition of $\mathcal H$, the vector $v_2$ corresponding to $F'_2$ satisfies $S_-(v_2)= [1,l]$. That is, $S_-(v_2)\subset S_+(v_1)$. Interchanging the roles of $F_1,F_2$, we get that $S_-(v_1)\subset S_+(v_2)$. Moreover, $S_+(v_1)\cap S_+(v_2) = \emptyset$. This means that $v_1$ and $v_2$ are antipodal, a contradiction.

Therefore, the family $\mathcal H'\cap \mathcal F'$ is intersecting. It is proven in \cite{Ka2} that in this case $|\mathcal H'\cap \mathcal F'|\le k$, but we sketch the proof of this simple fact here for completeness. Take a set $H\in \mathcal H'\cap \mathcal F'$. Then the $2k-2$ sets from $\hh'$ that intersect $H$ can be split into pairs of disjoint sets. We can take only one set from each pair.
\end{proof}

The rest of the argument is a standard averaging argument. Let us count in two ways the number of pairs (permutation $\sigma$, a vector from $\mathcal H(\sigma)\cap \mathcal F$). On the one hand, each vector from $\ff$ is counted $n k!l!(n-k-l)!$ times. On the other hand, for each permutation, there are at most $k$ pairs by Lemma~\ref{lem3}. Therefore,
$$|\ff|\,n\,k!\,l!\,(n-k-l)!\le kn!\ \ \ \Leftrightarrow \ \ \ |\ff|\le \frac kn\, \frac {n!}{k!\, l!\,(n-k-l)!} = \frac kn\, |V(n,k,l)|.$$

\end{document}